\documentclass[a4paper, 10pt, reqno]{amsart}
\usepackage[colorlinks=true]{hyperref}
\usepackage{amssymb, amscd, amsthm, mathtools}
\usepackage{url,amssymb,enumerate,colonequals}
\usepackage[margin = 1.3in]{geometry}
\usepackage{pdfsync}
\usepackage{multirow}
\usepackage[usenames,dvipsnames]{xcolor}
\usepackage{url}
\usepackage{microtype}
\usepackage{tikz-cd} 
\usepackage[all]{xy}
\usetikzlibrary{matrix}
\usepackage{float}
\usepackage{ulem}

\newcommand{\BA}{\mathbb{A}}

\newcommand{\BC}{\mathbb{C}}
\newcommand{\R}{\mathbb{R}}
\newcommand{\F}{\mathbb{F}}
\newcommand{\bF}{\mathbf{F}}
\newcommand{\Q}{\mathbb{Q}}
\newcommand{\Z}{\mathbb{Z}}
\newcommand{\Fbar}{{\overline{\F}}}

\newcommand{\calO}{\mathcal{O}}

\newcommand{\cl}{\mathrm{Cl}}
\newcommand{\fa}{\mathfrak{a}}
\newcommand{\fb}{\mathfrak{b}}
\newcommand{\fd}{\mathfrak{d}}

\newcommand{\fp}{\mathfrak{p}}
\newcommand{\cc}{\mathfrak{c}}

\newcommand{\fP}{\mathfrak{P}}
\newcommand{\fq}{\mathfrak{q}}

\newcommand{\vp}{\mathring{p}}

\newcommand{\tp}{\tilde{p}}

\newcommand{\leg}[2]{\big(\frac{#1}{#2}\big)}

\DeclareMathOperator{\End}{\mathrm{End}}
\DeclareMathOperator{\Gal}{\mathrm{Gal}}
\DeclareMathOperator{\Trd}{\mathrm{Trd}}

\DeclareMathOperator{\disc}{\mathrm{disc}}

\DeclareMathOperator{\Hom}{\mathrm{Hom}}

\DeclareMathOperator{\Ell}{\mathrm{Ell}}
\DeclareMathOperator{\Iso}{\mathrm{Iso}}

\DeclareMathOperator{\rk}{\mathrm{rk}}

\newtheorem{theorem}{Theorem}[section]
\newtheorem{lemma}[theorem]{Lemma}
\newtheorem{corollary}[theorem]{Corollary}
\newtheorem{proposition}[theorem]{Proposition}

\theoremstyle{definition}

\newtheorem{remark}[theorem]{Remark}

\numberwithin{equation}{subsection}

\usepackage[shortlabels, inline]{enumitem}
\SetEnumerateShortLabel{a}{\textup{(\alph*)}}
\SetEnumerateShortLabel{A}{\textup{(\Alph*)}}
\SetEnumerateShortLabel{1}{\textup{(\arabic*)}}
\SetEnumerateShortLabel{i}{\textup{(\roman*)}}
\SetEnumerateShortLabel{I}{\textup{(\Roman*)}}

\begin{document}
\title[Factorization of Hilbert class polynomials over $\F_p$]{Factorization of Hilbert class polynomials over prime fields}
	
\author{Jianing Li}
\address{CAS Wu Wen-Tsun Key Laboratory of Mathematics, University of Science and Technology of China,  Hefei, Anhui 230026,  China}
\email{lijn@ustc.edu.cn}

\author{ Songsong Li}
\address{School of Electronic Information and Electrical Engineering, Shanghai Jiao Tong University, Shanghai 200240,  China}
\email{songsli@sjtu.edu.cn}

\author{Yi ouyang}
\address{CAS Wu Wen-Tsun Key Laboratory of Mathematics, University of Science and Technology of China,  Hefei, Anhui 230026,  China}
\email{yiouyang@ustc.edu.cn}


\maketitle	
\begin{abstract} Let $D$ be a negative integer congruent to $0$ or $1\bmod{4}$ and $\mathcal{O}=\mathcal{O}_D$ be the corresponding order of $ K=\Q(\sqrt{D})$. The  Hilbert class polynomial $H_D(x)$ is the minimal polynomial of the $j$-invariant $ j_D=j(\mathbb{C}/\mathcal{O})$ of $\mathcal{O}$ over $K$. Let $n_D=(\mathcal{O}_{\Q( j_D)}:\Z[ j_D])$ denote the index of $\Z[ j_D]$ in the ring of integers of $\Q(j_D)$. Suppose $p$ is any prime. We completely determine the factorization of $H_D(x)$ in $\mathbb{F}_p[x]$ if either $p\nmid n_D$ or $p\nmid D$ is inert in $K$ and the $p$-adic valuation  $v_p(n_D)\leq 3$. As an application, we analyze the key space of Oriented Supersingular Isogeny Diffie-Hellman (OSIDH) protocol proposed by Col\`o and Kohel in 2019 which is the roots set of the Hilbert class polynomial in $\mathbb{F}_{p^2}$.
\end{abstract}

\section{Introduction}\label{sec: introduction}	

Let $D$ be a negative integer congruent to $0,1\bmod{4}$. Then $D$  is the discriminant of a unique order $\calO$ in an imaginary quadratic field $K$ (which is the field $\Q(\sqrt{D})$). The $j$-invariant of the complex elliptic curve $E_{\calO}=\BC/\calO$, denoted by $ j_D=j_{\calO}$, is an algebraic integer. Its minimal polynomial $H_D(x)\in \Z[x]$ over $K$ is called the Hilbert class polynomial, whose splitting field $L$ is  the ring class field of $\calO$.  
For a prime $p$ which does not split in $K$, the reduction of $E_{\calO}$ modulo prime ideals lying above $p$ in $L$ is supersingular over $\F_{p^2}$. This fact gives the close connection of Hilbert class polynomial and the isogeny-based cryptography. The goal of this paper to study the factorization of $H_D(x)$ modulo a prime $p$ and explore the application to the isogeny-based cryptography.  
 \subsection{Motivation}
The endomorphism rings of supersingular elliptic curves over $\Fbar_p$ are maximal orders in the quaternion algebra $B_{p,\infty}$ which is ramified  only at $p$ and $\infty$. Moreover, Deuring\cite{Deuring} proved that there is a one-to-one correspondence between isomorphism classes of maximal orders in $B_{p,\infty}$ and isomorphism classes of supersingular elliptic curves up to the action of $\Gal(\F_{p^2}/\F_p)$.

Computing endomorphism rings of supersingular elliptic curves is a central problem in isogeny-based cryptography. The best-known isogeny-based cryptographic protocol is Supersingular Isogeny Key Encapsulation(SIKE), which is a Round $3$ alternate candidate in the  NIST Post-Quantum Cryptography Standardization Project based on  Supersingular Isogeny Diffie-Hellman protocol(SIDH) proposed by De Feo and Jao\cite{SIDH} (also see \cite{DJP14}). Another popular protocol is  CSIDH (i.e. Commutative SIDH) proposed by Castryck, Lange, Martindale, Panny, and Renes\cite{CLMPR18} in 2018. CSIDH uses the action of an ideal class group on the set of supersingular elliptic curves defined over $\F_p$. In 2019, Col\`o and Kohel generalized CSIDH to OSIDH (i.e. Oriented SIDH), using a general ideal class group of an imaginary quadratic order $\calO$ and its action on the set of primitive $\calO$-oriented supersingular elliptic curves. One of the motivations for OSIDH is to enlarge its key space to $\Omega(p)$ by including all isomorphism classes of supersingular elliptic curves.   

The fundamental problem in the isogeny-based cryptography is finding isogenies between two supersingular  elliptic curves.  As pointed out in\cite{CPV20,EHL18}, computing isogenies can be reduced to the problem of computing the endomorphism rings of supersingular elliptic curves. One approach to compute the endomorphism rings is to find cycles in the isogeny graph of supersingular elliptic curves.
There are some works \cite{Kohel,GPS2016,EHL20} to analyze the algorithm complexities
based on different kinds of heuristic assumptions. Another approach is to construct the list of Deuring's correspondence by computing the elliptic curves whose endomorphism rings are isomorphic to given maximal orders in $B_{p,\infty}$.  

Dorman\cite{Dorman} and Ibukiyama\cite{Ibukiyama} described some isomorphic classes of maximal orders in $B_{p,\infty}$. More precisely, for any prime $q$ satisfying $q\equiv 3 \bmod 8$ and $\bigl(\frac{p}{q}\bigr)=-1$, Ibukiyama gave two kinds of maximal orders $\calO(q,p)$ and $\calO'(q,p)$, and proved that they are isomorphic to the endomorphism rings of some supersingular elliptic curves which are defined over $\F_p$. Given such an order $R=\calO(q,p)$ or $\calO'(q,p)$ in $B_{p,\infty}$, assume $j(R)\in\F_{p}$ is the unique $j$-invariant under Deuring's correspondence. Denote $R^{T}=\{2x-\Trd(x) \ \mid\ x\in R\}$. Then by \cite[Theorem 5.1]{CG14}, $d>4$ is represented optimally by $R^T$ with multiplicity $m$ if and only if $j(R)$ appears as a root of the Hilbert polynomial $H_{-d}(x)\in \F_p[x]$ with multiplicity $\epsilon m$, where $\epsilon=2$ if $p$ is ramified in $\Q(\sqrt{-d})$, and $1$ if $p$ is inert. Based on this fact, Chevyrev and Galbraith\cite{CG14} proposed an algorithm to determine $j(R)$ by computing the greatest common divisors of several Hilbert class polynomials over $\F_p[x]$. On the other hand, let $\calO=R\cap \Q(\sqrt{-q})$ and $D$ its discriminant. Then $\calO=\Z[\frac{1+\sqrt{-q}}{2}]$ or $\Z[\sqrt{-q}]$, and $D\in\{-q,-4q\}$. Castryck et al.\cite{CPV20} showed that there is exactly one root in $\F_p$ of $H_D(x)\bmod p$ if $D>-p$.  As a result, one can determine  $j(R)$ which is the unique $\F_p$-root of $H_D(x)$ modulo $p$.  However, it was shown in \cite[Theorem 1.3]{SYZ} that for any supersingular elliptic curve $E/\F_p$, the smallest prime $q$ such that $\End(E)\cong \calO(q,p)$ or $\calO'(q,p)$ is less than $10000p\log^6p$. For $D<-p$,  to our best knowledge,  the number of $\F_p$-roots of $H_D(x)$ is still unknown. 

Generally, let $\calO$ be an imaginary quadratic order of discriminant $D$. Suppose that $p$ is a prime which does not split in $K=\Q(\sqrt{D})$. By Deuring's reduction theorem\cite{Lang}, the roots of $H_D(x)$ modulo $p$ are supersingular $j$-invariants in $\F_{p^2}$. As we shall discuss in \S 5, the set of roots is just the key space of OSIDH if $D>-p$. In 2021, Xiao et al. \cite{GLY21} obtained the number of $\F_p$-roots of $H_D(x)$ modulo $p$ when $p\nmid D$ and $D>-\frac{4}{\sqrt{3}}\sqrt{p}$.  Motivated by these problems and their potential applications in the isogeny-based cryptography, we study the factorization of $H_D(x)$ over $\F_p[x]$.

\subsection{Our contributions} Let $D<0$ and $D\equiv 0,1\bmod{4}$. Let $K$ be the imaginary quadratic field $\Q(\sqrt{D})$, $\calO_K$  its ring of integers and $D_K$  its fundamental discriminant. Then $f=\sqrt{D/D_K}\in \Z$. The order $\calO=\Z+f\calO_K$ is the unique order in a quadratic field with discriminant $D$ (and conductor $f$).   Let $ j_D=j(E_{\calO})$,  $M=\Q( j_D)$, $L=K( j_D)$ and $n_D =(\calO_M:\Z[ j_D])$. Then $L$ is the splitting field of $H_D(x)$ over $K$ and is the ring class field of $\calO$ over $K$ as well. 

Assume $p$ is a prime.  We first study the prime factorization of $p\calO_M$ using properties of general dihedral groups $\Gal(L/\Q)$ and class field theory. Then we get the  factorization of Hilbert class polynomial $H_D(x)$ over $\F_p[x]$ if $p\nmid n_D$ in \S4. In the case  $p\mid  n_D $ and $\big(\frac{D_K}{p}\big)\neq 1$, we propose an approach to analyze the factorization of $H_D(x)$ modulo $p$, and determine the multiplicities of its irreducible factors if $p\nmid D$, $D>-p^3$ and $v_p( n_D )\leq 3$. 

By combining the selection of parameters proposed by Onuki\cite{Onuki} in 2021 for OSIDH protocol to work, we note that the roots set of the Hilbert class polynomial over $\F_{p^2}$ is just the key space of OSIDH up to $\Fbar_p$-isomorphism. And we point out the size of its key space is less than $O(\sqrt{p}\log p)$  which is contrary to the common belief in \cite{CK19} that it could be $\Omega(p)$.

\subsection{Paper organization} In section 2 we recall necessary backgrounds on number theory and elliptic curves. In section 3, we use results in general dihedral group to compute the factorization of $p\calO_M$. In section 4, we study the factorization of $H_D(x)$ over $\F_p[x]$, especially for primes $p$ dividing the discriminant of $H_D(x)$. In the last section, we give some analysis about the key space of OSIDH protocol.

\subsection*{Acknowledgments} The authors would like to thank Chaoping Xing for helpful comments on this article. Research is partially supported by Anhui Initiative in Quantum Information Technologies (Grant No. AHY150200), National Key R\&D Program of China (Grant NO. 2020YFA0712300) and NSFC (Grant No. 123031011).


\section{Preliminaries}

\subsection*{Conventions}
We shall adopt the following conventions:
\begin{enumerate}[1]
\item $p$ is always a prime, and the symbol $\left(\frac{a}{p}\right)$ is  the Legendre-Kronecker symbol, i.e., the Legendre symbol for odd $p$ and
\[ \left(\frac{a}{2}\right)=0, 1, -1\ \text{if}\ 2\mid a,\ a\equiv 1\bmod{8},\ a\equiv 5\bmod{8}\ \text{respectively}. \] 		
\item For a number field $\bF$, let $\calO_\bF$ be its ring of integers and $D_\bF$ be its discriminant. 
\item Suppose $\bF/\bF'$ is an extension of number fields, $\fp$  is a prime ideal of $\calO_{\bF'}$ above $p$ and   $\fP$  is a prime ideal of $\calO_\bF$ above $\fp$.  
\begin{enumerate}[i]
	\item Let $e_\fP(\bF/\bF')$  be the ramification index and $f_\fP(\bF/\bF')$ the inertia degree of $\fP$ in $\bF/\bF'$. 
	
	\item If $\bF/\bF'$ is a Galois extension, let $D_\fP(\bF/\bF')$ be the decomposition group and $I_\fP(\bF/\bF')$  the inertia group of $\fP$; in this case,  $e_{\fp}(\bF/\bF') = e_{\fP}(\bF/\bF')$ and $f{\fp}(\bF/\bF')=f_{\fP}(\bF/\bF')$ are independent of $\fP$ above $\fP$.
	
	\item In the case that $\bF/\Q$ and $\bF'/\Q$ are both Galois,   $e_p(\bF/\bF')=e_{\fP}(\bF/\bF')$ and   $f_p(\bF/\bF')=f_{\fP}(\bF/\bF')$ are independent of $\fP$ above $p$.
\end{enumerate}

\item For a prime $\fp\subset \calO_\bF$, the degree of $\fp$ is
 \[ \deg(\fp):=f_\fp(\bF/\Q).\] 
\end{enumerate}

\subsection{Orders in imaginary quadratic fields}\label{sec:class field} 

Let $\calO$ be an order of discriminant $D$ in an imaginary quadratic field $K$,  then $K=\Q(\sqrt{D})$, the conductor $f=[\calO_K:\calO]$, $\calO=\Z+f\calO_K$ and $D=f^2 D_K<0$ satisfying $D\equiv 0,1\bmod{4}$. On the other hand, let $D<0$ and $D\equiv 0,1\bmod{4}$. Let $K=\Q(\sqrt{D})$. Then $f=\sqrt{D/D_K}$ is an integer and $\calO=\Z+f \calO_K$ is an order of discriminant $D$ in $K$ with $f$ its conductor. Hence the order $\calO$, the discriminant $D$ and the pair $(K,f)$ are mutually determined. We shall fix this correspondence from now on.

Let $I(\calO)$ be the set of all proper fractional ideals of $\calO$, then $I(\calO)$ is a group under multiplication and contains $P(\calO)$, the set of principal fractional $\calO$-ideals, as a subgroup. The ideal class group of $\calO$ is the quotient group $\cl(\calO)=I(\calO)/P(\calO)$, which is a finite abelian group. Let $h_D=h_{\calO}$ be its class number. 

The following well-known facts can be found in \cite[\S 7.]{Cox89}:
\begin{proposition}\label{prop:rclgp0}
 Let $D$, $\calO$, $f$ and $K$ be given as above. 
	
	\begin{enumerate}[1]
		\item The class number $h_D=h_{\calO}$ of $\cl(\calO)$ is given by
		\begin{equation} \label{eq:classno}
			h_D= \frac{h_K f}{[\calO_K^\times: \calO^\times]}\cdot \prod_{p\mid f} \left (1-\Bigl(\frac{D_K}{p}\Bigr)\frac{1}{p} \right ). \end{equation}
			
	\item 		Let $I_K(f)$ be the subgroup of fractional ideals of $K$ generated by primes not dividing $f$, and $P_{K,\Z}(f)$ be the subgroup of $I_K(f)$ given by 
	\[  P_{K,\Z}(f)=\{(\alpha) \in I_{K}(f)  \mid \alpha \in \calO_K, \alpha \equiv a \bmod f\calO_K \text{ with } a\in \Z, (a,f)=1 \}.\] 
	Then there is a canonical isomorphism:
	\begin{equation}\label{eq:clgp_iso}
		I_{K}(f)/P_{K,\Z}(f) \cong \cl(\calO),\quad [I] \mapsto [I\cap \calO].
	\end{equation}	
	\end{enumerate}
\end{proposition}

\subsection{The ring class field and its maximal real subfield}
The isomorphism $\cl(\calO)\cong I_{K}(f)/P_{K,\Z}(f)$ in Proposition~\ref{prop:rclgp0}(2) means that 
 $\cl(\calO)$ is a generalized ideal class group, with $P_{K,\Z}(f)$ a congruence subgroup for the modulus 
$f\calO_K$. By the existence theorem of class field theory, this data determines a unique abelian
extension $L$ of $K$, which is called  the ring class field of $\calO$. Note that the
ring class field of $\calO_K$ is just the Hilbert class field.
The basic properties of the ring class field $L$ are that any primes that ramify in
$L$ divide $f$, and that the Artin reciprocity map induces an isomorphism
  \begin{equation}\label{eq:artin}
     \Gal(L/K)\cong \cl(\calO)
  \end{equation}
Let  $M=L^+:=L\cap \R$ be the maximal real subfield of $L$.

We give an idelic description of the class group $\cl(\calO)$ of $\calO$. Let $\BA^\times_K$ be the idele group of $K$.
For any $p$, set 
\[ 
\calO_{p,f} = \Z_p \otimes \calO  \quad \text{ and } \quad U_f = \prod_{p}\calO^\times_{p,f}\times \BC^\times \subset \BA^\times_K.\] If $v$ is a prime of $K$, let $K_v$ denote the completion of $K$ at $v$ and $\calO_{v}$ denote the ring of integers in $K_v$. We let $\fp_K$ denote a prime of $K$ lying above $p$.

\begin{proposition}\label{prop:rclgp1}
 Let $D$, $\calO$, $f$ and $K$ be given as above. 
\begin{enumerate}[1]
\item  If $p\nmid f$, then \begin{equation*}
\calO^\times_{p,f} \cong (\Z_p \otimes \calO_K)^\times  \cong \begin{cases}
 \calO^\times_{\fp_K}  & \text{ if } p \text{ does not split in } K,\\
 \Z^\times_p\times \Z^\times_p & \text{ if } p \text{ splits in }  K.
\end{cases}
\end{equation*}
\item If $p\mid f$, then
\begin{equation*}
\calO^\times_{p,f} \cong 
\begin{cases}
\{ x \in \calO^\times_{\fp_K}\mid x\equiv a \bmod f\calO_{\fp_K} \text{ for some } a \in \Z\} & \text{ if } p \text{ does not split in }K,\\
\{ (x,y)\in \Z^\times_p\times\Z^\times_p\mid x\equiv y \bmod f\Z_p  \} & \text{ if } p \text{ splits in }K.
\end{cases}
\end{equation*}

\item The Artin map gives a canonical isomorphism
\begin{equation*}
\BA^\times_K/K^\times U_f \cong  \Gal(L/K).
\end{equation*}
Thus $K^\times N_{L/K}\BA^\times_L=K^\times U_f$.
\end{enumerate}
\end{proposition}
\begin{proof}
By the exact sequence 
\begin{equation*}
0\to \calO \to \calO_K \to \calO_K/\calO\to 0,
\end{equation*}
Then $\calO\otimes \Z_p=\calO_K \otimes \Z_p$ for $p\nmid f$. Thus
the assertion (1) follows from the canonical isomorphism:
\begin{equation*}
\calO_K \otimes \Z_p \cong \begin{cases}
\calO_{\fp_K} & \text{ if } p \text{ does not split}, \\
\Z_p\times \Z_p & \text{ if } p \text{ splits}.
\end{cases}
\end{equation*}

 For (2), take $w\in \calO_K$ such that $\calO_K=\Z + \Z w$. Then $\calO = \Z + \Z fw$ and $\calO \otimes \Z_p \cong \Z_p +\Z_p f w$. Since $p\mid f$, an element $x=u+vf w \in\calO \otimes \Z_p$ with $u,v\in \Z_p$ is in $(\calO\otimes \Z_p)^\times$ if and only if $p\nmid u$. The latter is also equivalent to that $ x\equiv a\bmod f(\calO \otimes \Z_p)$ for some $a\in \Z$ and $(a,f)=1$. This proves (2). 
 
 For (3), we let 
\[
\BA^\times_{K,\Z}(f) = \{ (x_v)_v \in \BA^\times_K\mid x_v \equiv a \bmod f\calO_v \text{ with } a\in \Z \text{ and }(a,f)=1\}.
\]
Consider the canonical map $\pi: \BA^\times_K \to I_K$, where $I_K$ is the group of fractional ideals of $K$. This map induces an isomorphism 
\begin{equation*}
\BA^\times_{K,\Z}(f)/U_f \to I_K(f).
\end{equation*}
Note that $\pi(K^\times \cap \BA^\times_{K,\Z}(f))$ is equal to $P_{K,\Z}(f)$. Thus we obtain an isomorphism 
\begin{equation*}
\BA^\times_{K,\Z}(f)/(K^\times \cap \BA^\times_{K,\Z}(f)) U_f \cong I_{K,\Z}(f)/P_{K,\Z}(f).
\end{equation*}
The inclusion $\BA^\times_{K,\Z}(f) \to \BA^\times_K$ induces an isomorphism 
\begin{equation}
\BA^\times_{K,\Z}(f)/(K^\times \cap \BA^\times_{K,\Z}(f)) U_f \cong \BA^\times_K/K^\times U_f.
\end{equation}
The injection is a direct verification and the surjection is by the approximation theorem. By \eqref{eq:artin} and Proposition~\ref{prop:rclgp0}(2), the proof of Proposition~\ref{prop:rclgp1} is complete. \end{proof}

If $p\mid f$, we let $f^{(p)}$ be the prime-to-$p$ part of $f$ and $\calO^{(p)}=\Z+f^{(p)} \calO_K$ be the order of $K$ of conductor $f^{(p)}$. Then the corresponding discriminant $D^{(p)}=(f^{(p)})^2 D_K$. Let $L^{(p)}$ be the corresponding ring class field of $\calO^{(p)}$ and  $M^{(p)}:=L^{(p)}\cap \R$ be its maximal real subfield.

\begin{proposition}\label{prop:totrami}  Suppose $f=p^k f^{(p)}$ with $p\nmid f^{(p)}$ and $k\geq 1$. Then
	\begin{enumerate}[1]
		\item  $L$ is an extension of $L^{(p)}$ of degree
		\begin{equation} \label{eq:classno2}
		h^{(p)}_D=h_D/h_{D^{(p)}}=\frac{p^k}{[(\calO^{(p)})^\times:\calO^{\times}]}\left(1-\Bigl(\frac{D_K}{p}\Bigr)\frac{1}{p}
			\right). 
		\end{equation} 
		Furthermore,  $L/L^{(p)}$ is totally ramified at every prime lying above $p$.
	 \item $M$ is an extension of $M^{(p)}$ of degree $h^{(p)}_D$ and 
	 is totally ramified at every prime lying above $p$.
	\end{enumerate}
\end{proposition}

\begin{proof}
(1) By Proposition~\ref{prop:rclgp1}, we have $U_f\subset U_{f^{(p)}}$. According to class field theory, we have $L\supset L^{(p)}$ and $\Gal(L/L^{(p)})$ is generated by the inertia groups of the primes of $K$ lying above $p$. Thus, if $p$ does not split in $K$, then $L/L^{(p)}$ is totally ramified at each prime lying above $p$. Now assume that $p$ splits in $K$, say $p\calO_K = \fp_K \fp'_K$. Then we claim
$I_{\fp_K}(L/K) =I_{\fp'_K}(L/K)$. 
By class field theory, it suffices to show that 
\begin{equation}\label{eq:inert0}
\calO^\times_{\fp_K} (K^\times U_f) = \calO^\times_{\fp'_K} (K^\times U_f).
\end{equation}
Given $a=(1,\cdots, 1,\underset{\fp_K}{a},1\cdots, 1)\in \calO^\times_{\fp_K}\subset \BA^\times_K$, we have by Proposition~\ref{prop:rclgp1}
\begin{equation}
a = (1,\cdots, 1, \underset{\fp'_K}{a^{-1}},1\cdots, 1) (1,\cdots, 1,\underset{\fp_K}{a},\underset{\fp'_K}{a}\cdots, 1)  \in \calO^\times_{\fp'_K} U_f.
\end{equation}
From this, we conclude that \eqref{eq:inert0} holds whence the claim is proved. Therefore when $p$ splits, $L/L^{(p)}$ is also totally ramified at each prime lying above $p$. The formula \eqref{eq:classno2} follows from Proposition~\ref{prop:rclgp0}(1).

(2)  Since $L^{(p)}$ is not contained in $M=L\cap \R$, we have $M\supset M^{(p)}$ and $[M:M^{(p)}]=h_D^{(p)}$. If a prime $\fq$ of $M^{(p)}$ above $p$ is ramified in $L^{(p)}/M^{(p)}$, then $\fq$ must be totally ramified in $L/M^{(p)}$ whence totally ramified in $M/M^{(p)}$; if $\fp'$ is unramified in $L^{(p)}/M^{(p)}$, noting that $L=L^{(p)}M$, it follows that every prime of $M$ lying above $\fq$ is unramified in $L$. Thus by comparing the ramification index in the extension $L/M^{(p)}$, we conclude that $M/M^{(p)}$ is totally ramified at every prime above $p$. \end{proof}



\subsection{The genus field and its maximal real subfield}\label{sec:genus}

Fix $D$ and $\calO$. We denote by $F$  the genus field  of $\calO$ and $F^+=F\cap \R$ its maximal real subfield, which means that $F$ is the intermediate field of $L/K$ fixed by $\cl(\calO)^2=\Gal(L/K)^2$, i.e.  
\begin{equation}
	F=L^{\cl(\calO)^2},\quad \Gal(L/F)=\Gal(L/K)^2 = \cl(\calO)^2.
\end{equation} 
Define $\mu\geq 1$ to be the integer such that 
\begin{equation}\label{eq:mu}
	\mu=\mu_D=\log_2[F:\Q]= \rk_2 \cl(\calO)+1.
\end{equation}
Genus theory tells us that  $\Gal(F/\Q)\cong (\Z/2\Z)^\mu$ and $\Gal(F/K)\cong \Gal(F^+/\Q)\cong (\Z/2\Z)^{\mu-1}$.

Suppose $\{p_1,\cdots, p_r\}$ is the set of odd prime factors of $D$ such that the first $m$ are $\equiv 1\bmod 4$ and the rest $\equiv 3\bmod 4$. 
Set
 \begin{align} &\vp_i=p_i\qquad (\text{if}\ p_i\equiv 1\bmod{4});\\
 	&\vp_i=\begin{cases} -D/p_i, & \text{if}\ D\ \text{or}\ D/4\ \text{or}\ D/8\equiv 1\bmod{4} \\ 2p_i, &\text{if}\ D/8\equiv 3\bmod{4}\\ p_i, &\text{if}\ D/4\equiv 0, 3\bmod{4} 
 	\end{cases}\qquad (\text{if}  \ p_i\equiv 3\bmod{4}); \\
  &\vp_0=\begin{cases}
  	2, &\text{if}\ D/8\equiv 0, 1\bmod{4};\\ 1, & \text{if otherwise}.
  \end{cases}
 \end{align}

\begin{proposition}\label{prop:genus_field1} The fields $F^+$ and $F$ are given as follows:
\[ F^+=\Q(\sqrt{\vp_0},\sqrt{\vp_1},\cdots, \sqrt{\vp_r} ) \quad \text{ and } \quad F = F^+ K=F^+ (\sqrt{D_K}).  \]
Consequently, 
\begin{enumerate}[1]
	\item $\mu=1$, i.e. $F^+=\Q$ and $F=K$ if and only if $D\in \{-4, -8, -16,-p^{2k+1}, -4p^{2k+1}\mid p\equiv 3\bmod{4}, k\in \Z_{\geq 0}\}$;
	\item  For any prime $p$,  $f_p(F^+/\Q)\leq f_p(F/\Q)\leq 2$.
\end{enumerate}
\end{proposition}

\begin{proof}
See \cite[Theorem 2.3.23]{Cohn94}.
\end{proof}

We now assume a rational prime $p$ always splits completely in $\Q$. Applying Proposition~\ref{prop:genus_field1}, we obtain the following results, which will be used in Theorem~\ref{them:prime factorization}, after some computation:
\begin{lemma}\label{lem:p0} 
Suppose that $p\nmid f$ and $p$ does not split in $K$. Let  $\fp_K$ denote the unique prime of $K$ above $p$. Then 
\begin{enumerate}[1]	
	\item $\fp_K \cap \calO$ is principal if and only if $p$ is either inert in $K$ or $p$ is ramified in $K$ and $D\in \{-p,-2p, -4p\}$. If $p$ is inert in $K$, then $p$ splits completely in $F^+$ if and only if $\leg{\vp_1}{p} =  \cdots \leg{\vp_r}{p} = 1$ if $p>2$, or $\{p_i\bmod{8}\}=\{1,3\}$ or $\{1,7\}$ if $p=2$.
	
	\item $\fp_K \cap \calO$ is not principal  if and only if $p\mid D_K$ and $D\notin\{ -p, -2p,-4p\}$.
\end{enumerate}  
	Assume furthermore  $p\mid D_K$ and $D\notin\{ -p, -2p,-4p\}$, equivalently  $\fp_K \cap \calO$ is not principal. 	\begin{enumerate}[resume*]

		\item $p$ is unramified in $F^+$ if and only if $4^2\nmid D$, $p_i\equiv 1\bmod{4}$ for all $p_i\mid D$ but $p_i\nmid 2p$ and  $p\not\equiv 1\bmod{4}$. In this case, $p$ splits completely in $F^+$  if and only if (i) $\leg{p_i}{p}=1$ if  either $p\equiv 7\bmod{8}$ or $p\equiv 3\bmod{8}$ and $D$ or $D/4\equiv 1 \bmod 4$, or (ii) $p_i\equiv 1\bmod{8}$ if $p=2$.
	
		\item $p$ is ramified in $F^+$ if and only if  either (i)  $p\equiv 1\bmod 4$ or (ii) $D$ is not of the form $-2^a p_1\cdots p_m p$ with $a=0,2,3$,  $p_i\equiv 1\bmod 4$ and $p\equiv 3\bmod{4}$ or (iii) $D$ has a prime factor $\equiv 3\bmod{4}$ and $p=2$.  In this case,  $f_p(F/F^{+}) = 2$ if and only if (i) $\leg{\tp_i}{p}={1}$ for all $0\leq i\leq r$ where $\tp_i$ is the prime-to-$p$ part of $\vp_i$ and $\leg{D_K/{p}}{p}= -1$ if $p>2$, or (ii) all odd prime factors of $D$ are $1$ or $3\bmod{8}$ if $p=2$.	
		\end{enumerate}
\end{lemma}
\begin{remark} The condition $D\in \{-p,-2p, -4p\}$ means  $D\in \{-p,-4p\}$ if $p\equiv 3\bmod{4}$,  $D\in \{-4,-8\}$ if $p=2$  and $D=-4p$ if $p\equiv 1\bmod{4}$, hence  $F^+=\Q$ if $p\not\equiv 1\bmod{4}$.
\end{remark}

\subsection{Elliptic curves with complex multiplication}

For a lattice  $\Lambda \subseteq \BC$, let $E_{\Lambda}$ be the elliptic curve over $\BC$ such that $E_\Lambda(\BC) \cong \BC/\Lambda$. Then $E_{\Lambda}\cong E_{\Lambda'}$ (i.e. $j(E_{\Lambda})=j(E_{\Lambda'})$) if and only if $\Lambda=\lambda\Lambda'$ for some $\lambda \in \BC^{\times}$ (i.e. $\Lambda$ and $\Lambda'$ are homothetic).

Assume $D<0$ and  $D\equiv 0, 1\bmod{4}$. Let $\calO$ be the corresponding order.  Then the endomorphism ring  $\End(E_{\calO})$ of $E_{\calO}=\BC/\calO$ is nothing but $\calO$. Set
\begin{equation} j_D=j_{\calO}:=j(E_{\calO}). \end{equation}
Set
\begin{equation} \label{eq:ell(O)}
\Ell(\calO):= \{j(E)\mid \End(E)\cong \calO\}\ \ (= \{E\mid \End(E)\cong \calO\}/\sim). 
\end{equation}
The Hilbert class polynomial $H_D(x)$ is defined as
\begin{equation} H_D(x)=H_{\calO}(x):=\prod_{j(E) \in \Ell(\calO)}(x-j(E)).\end{equation}
The theory of complex multiplication tells us (see  ~\cite{ATAEC})
\begin{theorem} Given $D$ and $\calO$. 
	\begin{enumerate}[1]
		\item The polynomial $H_D(x) \in \Z[x]$ is the minimal polynomial of $ j_D$ over $K$ of degree $h_D$, whose conjugates form exactly the set  $\Ell(\calO)=\{j(E_{\mathfrak{b}})\mid [\mathfrak{b}]\in \cl(\calO)\}$. 
		\item The field  $K( j_D)$ is the splitting field  of $H_D(x)$ over $K$
		and  is the ring class field $L$ of $\calO$ over $K$. 
		\item The action of $\cl(\calO)$  on  $\Ell(\calO)$ by $[\mathfrak{a}]j(E_{\mathfrak{b}})=j(E_{\mathfrak{a}^{-1}\mathfrak{b}})$ and  the Galois action of $\Gal(L/K)$  on  $\Ell(\calO)$  are compatible under the Artin map 
		\begin{equation}\label{eq:artin1}
		\theta: \Gal(L/K)\cong \cl(\calO),
		\end{equation} both are free and transitive. 
	\end{enumerate}  	
\end{theorem}

For any $j$-invariant $j\in\Ell(\calO)\subset L$, we can take an elliptic curve $E/L$ such that $j(E)=j$. Then $\End(E)=\End_L(E)\cong \calO$ as $K\subset L$. Let $[\cdot]_E: \calO\rightarrow \End(E)$ be an isomorphism such that $(E,[\cdot]_E)$ is normalized(see \cite{ATAEC}). Then for any $E'\in \Ell(\calO)$ and a non-constant isogeny $\varphi: E\rightarrow E'$, we have $[\cdot]_{E'}=\varphi_{*}([\cdot]_E):=\frac{1}{\deg \varphi}\varphi\circ [\cdot]_E\circ \hat{\varphi}$.

\subsection{Oriented supersingular elliptic curves}\label{sec:osidh}
In this subsection, we briefly recall some definitions and known results about oriented supersingular elliptic curves in \cite{Onuki} which will be used in \S 5. Suppose $p$ does not split in $K=\Q(\sqrt{D})$. Let $L$ be the ring class field of $\calO$. We assume $\fP$ is a prime in $L$ above $p$ such that every elliptic curve in $\Ell(\calO)$ has a good reduction at $\fP$(otherwise, we can take a finite extension $L'/L$ and such a prime in $L'$ as $\Ell(\calO)$ is finite). Then $\bar{E} =E \bmod \fP$ is supersingular by Deuring's reduction theorem\cite[Chapter 13, Theorem 12]{Lang}. Define a map $\rho$ by the reduction modulo $\fP$ as 
\[\rho: \Ell(\calO) \rightarrow \rho(\Ell(\calO)),\ E\rightarrow (\bar{E},[\cdot]_{\bar{E}}),\]
where $[\alpha]_{\bar{E}}=[\alpha]_E \bmod \fP$ for any $\alpha\in \calO$. The map $[\cdot]_{\bar{E}}$ induces a ring homomorphism:
\[\iota:\ K \hookrightarrow \End(\bar{E})\otimes \Q.\]
Then a pair $(\bar{E},\iota)$ is called a $K$-oriented elliptic curve; furthermore, it is $\calO$-oriented(resp. primitive $\calO$-oriented) if $\iota(\calO)\subset \End(\bar{E})$(resp. $\iota(\calO)=\End(\bar{E})\cap\iota(K)$).

Moreover, suppose $\fa$ is a proper integral ideal of $\calO$ which is prime to $p$ and $E\in \Ell(\calO)$. Let $E'=[\fa]*E$. Then there is an
isogeny $\varphi: E\rightarrow E'$ with kernel $E[\fa]=\bigcap_{\alpha\in \fa}\ker([\alpha]_E)$\cite[Chapter 2]{ATAEC}. By the reduction modulo $\fP$, we have $\phi$ corresponds to $\bar{\varphi}: \bar{E}\rightarrow \bar{E'}$
whose kernel is $\bar{E}[\fa]:=\bigcap_{\alpha\in \fa}\ker([\alpha]_{\bar{E}})$. Thus there is an action of $\cl(\calO)$ on $\rho(\Ell(\calO))$ given by
\[[\fa]* (\bar{E},\iota) = (\bar{E'},\bar{\varphi}_{*}(\iota)),\] 
where $\bar{\varphi}_{*}(\iota)(\alpha)=\frac{1}{\deg(\bar{\varphi})}\bar{\varphi}\circ\iota(\alpha)\circ\hat{\bar{\varphi}}$.  The following lemma can be found in \cite{Onuki}.
\begin{lemma}
	Let $K$  be an imaginary quadratic field such that $p$ does not split in $K$, and $\calO$ an order in $K$ such that $p$ does not divide the conductor of $\calO$. Then the ideal class group $\cl(\calO)$ acts freely and transitively on $\rho(\Ell(\calO))$.
\end{lemma}

\section{Factorization of $p\calO_M$}

In this section, we fix $D$ and $\calO$, and hence $K, f$  as in \S~\ref{sec:class field}, and $F,F^{+},\mu$ as in \S~\ref{sec:genus}. Then $L=K( j_D)$ is the ring class field  of $\calO$, and $M=L\cap \R=\Q( j_D)$ is the maximal real subfield of $L$. The main goal of this section is to explicitly describe the prime ideal factorization of $p \calO_M$ for any $p$ which does not split in $K$. The case when $p$ splits in $K$, which is simpler, will also be described in the end of this section.

Before stating our result, we need to know more about the Galois extension $L/\Q$. Let $G=\Gal(L/\Q)$. Let $\tau$ be the nontrivial element in $\Gal(L/M)$ (i.e. the complex conjugation). Identifying $\Gal(L/K)=\cl(\calO)$ via the isomorphism \eqref{eq:artin}, then we have (see \cite[\S 9]{Cox89})
\begin{equation}\label{eq:groupG}
	G=\cl(\calO)\rtimes \langle \tau \rangle, \quad \tau \sigma \tau =\sigma^{-1} \text{ for } \sigma\in \cl(\calO). 
\end{equation}
Then $ \Gal(L/F)=\cl(\calO)^2$ and $\Gal(L/F^{+})=\cl(\calO)^2\rtimes \langle\tau\rangle.$
\[
\begin{tikzcd}[row sep=tiny, column sep = large]
	&  L=K(j_D)  \ar[dd, dash, "\cl(\calO)^2"']  \ar[dddd,dash, bend left=40, "\cl(\calO)"]\\
	M=\Q(j_D) \ar[ru, dash, "\tau"]          &      \\
	& F     \\
	F^{+}=F\cap \R  \ar[ru, dash, "\tau"] \ar[uu, dash]              &      \\
	&  K=\Q(\sqrt{D}) \ar[uu, dash]    \\
	\Q \ar[ru, dash,"\tau"] \ar[uu, dash]               &      \\                        
\end{tikzcd}
\]

\begin{lemma}\label{lem:G_L}
$(1)$ The conjugate class of $\tau$ in $G$ is $\cl(\calO)^2\tau=\{ \sigma^2 \tau \mid \sigma \in \cl(\calO)\}.$ 
		
$(2)$   The centralizer $C_G(\tau) =\{ x\in G\mid  x\tau x^{-1}=\tau \}$  of $\tau$ has order $\# C_G(\tau) =2^{\mu}.$
\end{lemma}
\begin{proof}
	Direct computation.
\end{proof}

In what follows, $\fp$ will always denote a prime of $M$ above $p$;  $\fP$ and $\fP'$ will denote primes of $L$ above $p$. We write $D_\fP=D_\fP(L/\Q)$ and $I_\fP=I_\fP(L/\Q)$ for simplification.

\begin{lemma}\label{lem:central} 
	\begin{enumerate}[1]
		\item  If there exists some $\fP'$ such that $D_{\fP'}=\langle \tau \rangle$, then 
		\begin{equation}\label{eq:central}
			\# \{ \fP \mid  \text{ } D_{\fP} = \langle \tau \rangle \} = \frac{\# C_G(\tau)}{\# D_{\fP'}} =2^{\mu-1}.
		\end{equation}
		
		\item  If there exists some $\fP'$ such that $I_{\fP'}=\langle \tau \rangle$, then 
		\begin{equation}\label{eq:central_inert}
			\# \{ \fP \mid I_{\fP} = \langle \tau \rangle \} = \frac{\# C_G(\tau)}{\# D_{\fP'}} =\frac{2^{\mu}}{\# D_{\fP'}}.
		\end{equation}
		
		\item Suppose $\sigma \in \cl(\calO)$, $\sigma\neq 1$, $\sigma^2=1$ and $\sigma\tau=\tau\sigma$. If there exists some $\fP'$ such that $D_{\fP'}=\{ 1, \sigma, \tau, \sigma\tau\}$ and $I_{\fP'}=\{1,\sigma\tau\}$, then 
		\begin{equation}\label{eq:central2}
			\# \{ \fP \mid D_{\fP} \supset \langle \tau \rangle, I_{\fP}\neq \langle \tau\rangle \} = \frac{\# C_G(\tau)}{\# D_{\fP'}} = 2^{\mu-2}.
	\end{equation} \end{enumerate}
\end{lemma}

\begin{proof} Note that $G$ acts on primes above $p$ in $L$ transitively,  $D_{x(\fP')}=xD_{\fP'} x^{-1}$ (resp. $I_{x(\fP')}=xI_{\fP'} x^{-1}$) for $x\in G$. 
	
	(1) and (2): We have  $D_{x(\fP')}=\langle\tau \rangle$ (resp. $I_{x(\fP')}=\langle\tau \rangle$) if and only if $x\in C_G(\tau)$. So $C_G(\tau)$ acts transitively on $\{ \fP \mid D_{\fP} = \langle \tau \rangle \}$ (resp. $\{ \fP \mid I_{\fP} = \langle \tau \rangle \}$), with  the stabilizer  $\langle \tau\rangle$ by definition. Then (1) and (2) follow from Lemma~\ref{lem:G_L}.

	(3) Write $\fP=x\fP'$ for some $x\in G$. Then $D_{\fP}=\{1, x\sigma x^{-1}, x\tau x^{-1}, x\sigma\tau x^{-1}\}$ and $I_{\fP}=\{1, x\sigma\tau x^{-1}\}$. Note that  $\sigma$ is not conjugate to $\tau$. Then
	the condition that $ D_{\fP} \supset \langle \tau \rangle$ and $I_{\fP}\neq \langle \tau\rangle$ is equivalent to the condition $x \tau x^{-1}=\tau$, i.e., $x\in C_G(\tau)$. Then (3) also follows from Lemma~\ref{lem:G_L}.
\end{proof}

\begin{theorem}~\label{them:prime factorization}
	Fix $D,\ \calO$ and hence the conductor $f$ and the fields $K, L, M, F$ and $F^{+}$. Assume $p$ does not split in $K$, and let $\fp_K$ be the unique prime of $K$ lying above $p$. 
	
$(A)$ Suppose $p\nmid f$. Then all primes of $M$ above $p$ are of degree $1$ or $2$. Let $h=h_D$ be given by \eqref{eq:classno} and $\mu$ be given by \eqref{eq:mu}. 
	\begin{enumerate}[1]
		\item   If $p \nmid D_K$, then 
		\[p\calO_M=\fp_1   \cdots\fp_g\]
	and $t$ of the primes $\fp_i$ are of degree $1$, where $t=2^{\mu-1}$ if $p$ splits completely in $F^+$ and $0$ otherwise (see Lemma~\ref{lem:p0}(1)) and $g=\frac{h+t}{2}$.  
		
	 If $D\in \{-p,-2p,-4p\}$, then  
			\[p\calO_M=\begin{cases} \fp^2_{1}\cdots \fp^2_{h/2},& \text{if}\ p\equiv 1\bmod{4}; \\ 
			\fp_1\fp^2_{2}\cdots \fp^2_{(h+1)/2},& \text{if}\ p\not\equiv 1\bmod{4};
			\end{cases}\]
where $\deg(\fp_i)=1$ for all $i$.
	
		\item Assume $p\mid D_K$ and $D\notin \{-p,-2p,-4p\}$.
		
		\begin{enumerate}[i]
			\item  If $p$ is unramified in $F^{+}$,  put  (see Lemma~\ref{lem:p0}(3))
			 \[ s=2^{\mu-2},\quad t= \begin{cases}
			 	2^{\mu-2}, &  \text{if } p \text{ splits completely in }F^{+};\\ 0, &\text{otherwise};
			 \end{cases} \quad  g = \frac{h+2t }{4} +2^{\mu-3}, \] 
			then \[p\calO_M=\fp_1\cdots\fp_{s} \fp_{s+1}^2\cdots \fp^2_{s+t} \cdots \fp^2_{g}  \]
			with $\fp_{s+1},\cdots, \fp_{s+t}$ being exactly the primes of $M$ above $p$ of degree $1$.
			
			\item If $p$ is ramified in $F^{+}$, put (see Lemma~\ref{lem:p0}(4))
			 \[ t=  \begin{cases}
			 	0, & \text{if}\ f_p(F/F^{+})=1;\\ 2^{\mu-2}, & \text{if}\ f_p(F/F^{+})=2;
			 \end{cases} \quad g=\frac{h+2t}{4}, \] 
			then \[p\calO_M=(\fp_1  \cdots\fp_{g})^{2} \]
			and the number of $i$'s that $\deg(\fp_i)=1$ is  $t$. 
		\end{enumerate}
	\end{enumerate}		

$(B)$ If $p\mid f$, then the factorization of $p\calO_{M^{(p)}}$ is given by $(A)$ and every prime above $p$ in $\calO_{M^{(p)}}$ is totally ramified in $M/M^{(p)}$ whose degree  $h_D^{(p)}$ is given by \eqref{eq:classno2}.
\end{theorem}

\begin{proof}
 It suffices to show (A). (B) follows from (A) and Proposition~\ref{prop:totrami}. Now assume conditions in (A). We shall use the following well-known formula 
 \begin{equation}\label{eq:efg}
  \sum_{\fp\mid p} e_{\fp}(M/\Q) \deg(\fp) =[M:\Q]=h.
 \end{equation}

		
(1) By Lemma~\ref{lem:p0}, the assumption in this case is equivalent to that $\fp_K \cap \calO$ is principal. Then $\fp_K$ splits completely in $L$ by \eqref{eq:artin}. It follows that for each $\fP$, we have  $\# D_{\fP} = e_p(L/\Q)f_p(L/\Q)=e_{p}(K/\Q) f_{p}(K/\Q)=2$. This implies 
 \[\begin{split} &e_{\fp}(M/\Q)\deg(\fp)\leq 2;\\ &e_{p}(F^{+}/\Q)f_p(F^{+}/\Q)\leq 2. 
 	\end{split}    \]
We divide the set of primes  above $p$ in $M$ into two subsets 
\begin{gather*}
X_1=\{\fp\mid e_{\fp}(M/\Q)\deg(\fp)=1 \},\\
X_2=\{\fp\mid e_{\fp}(M/\Q)\deg(\fp)=2 \}.
\end{gather*}
 Note that
 \begin{enumerate}[a]
 	\item For $\fp\in X_1$, $e_{\fp}(M/\Q)=\deg(\fp)=1$.
 	\item If $p$ is inert in $K$, then for $\fp\in X_2$,  $e_{\fp}(M/\Q)=1$ and $\deg(\fp)=2$. Indeed in this case, 
 	$e_p(L/\Q)=1$ and hence  $e_{\fp}(M/\Q)=1$ for all $\fp$.
 	\item If $p$ is ramified in $K$, then  for $\fp\in X_2$,   $\deg(\fp)=1$ and $e_{\fp}(M/\Q)=2$. Indeed in this case, $e_p(L/\Q)=2$, $f_p(L/\Q)=1$ and hence $\deg(\fp)=1$ all $\fp$. 
 \end{enumerate}
Let $t=|X_1|$  and  $g=|X_1|+|X_2|$ be the number of primes above $p$ in $M$. Then by \eqref{eq:efg}, $t+2(g-t)=h$ whence $g=\frac{h+t}{2}$. The proof of Theorem~\ref{them:prime factorization}(A)(1) is reduced to determining $t$.

Since $\Gal(L/M)=\langle \tau \rangle$, there is a bijection
\begin{equation}\label{eq:ef_to_D1}
	\{ \fP \mid D_{\fP}=\langle \tau \rangle \} \xleftrightarrow{1-1}{} X_1, \quad \fP \mapsto \fp=\fP \cap M.
\end{equation}
Note that $D_{\fP}(L/K)= D_\fP\cap \Gal(L/K) = \{1\}$, as $\fp_K$ splits completely in $L/K$. Thus from \eqref{eq:groupG},  $D_\fP=\langle \sigma \tau \rangle$ for some $\sigma \in \cl(\calO)$.

Suppose first that $p$ splits completely in $F^{+}$, i.e. $e_p(F^{+}/\Q)f_p(F^{+}/\Q)=1$. Then $D_\fP \subset \Gal(L/F^{+}) =\cl(\calO)^2\rtimes \langle \tau \rangle$ whence $D_\fP = \langle \sigma^2 \tau\rangle$ with $\sigma \in \cl(\calO)$. Then $D_{\fP'}=\langle \tau \rangle$ where $\fP'=\sigma^{-1}(\fP)$. It follows from Lemma~\ref{lem:central}(1) that $t=2^{\mu-1}$. 

	
Suppose next $e_p(F^{+}/\Q)f_p(F^{+}/\Q)=2$. Then $D_\fP \not \subset \Gal(L/F^{+})$ whence $D_\fP=\langle \sigma\tau \rangle$ for some $\sigma\in \cl(\calO)\setminus \cl(\calO)^2$. So $D_\fP$ is not conjugate to $\langle \tau \rangle$ by Lemma~\ref{lem:G_L}. This implies that the sets in \eqref{eq:ef_to_D1} are both empty and hence $t=0$. Finally, note that when $ D \in \{-p,-2p,-4p\}$
the value of $t$ is given by Proposition~\ref{prop:genus_field1}.	
	
\bigskip
	
(2) The assumption in this case is equivalent to that $\fp_K \cap \calO$ is not principal. 
Then by \eqref{eq:clgp_iso} $\fp_K\notin P_{K,\Z}(f)$ but $\fp^2_K=(p) \in P_{K,\Z}(f)$. By the isomorphism \eqref{eq:artin}, we have $e_{\fp_K}(L/K)=1$ and $f_{\fp_K}(L/K)=2$. Then 
 \begin{equation}\label{eq:ef2}
e_{p}(L/\Q) = e_{\fp_K}(K/\Q)=2, \quad\quad  f_{p}(L/\Q)=f_{\fp_K}(L/K)=2.
 \end{equation}
 In particular for each $\fp$, we have $\deg(\fp)\leq 2$, $e_{\fp}(M/\Q)\leq 2$. Moreover, there is no $\fp$ with $e_{\fp}(M/\Q) = \deg(\fp)=1$; otherwise we would have $e_{\fp}(L/M)f_{\fp}(L/M)=e_{p}(L/\Q)f_{p}(L/\Q)=4$ but this is impossible as $[L:M]=2$. So we can divide the set of primes above $p$ in $M$ into three subsets: 
 \begin{gather*}
 X_1=\{\fp\mid e_{\fp}(M/\Q)=1, \deg(\fp)=2 \}, \\
X_2=\{\fp\mid e_{\fp}(M/\Q)=2, \deg(\fp)=1 \},\\
X_3=\{\fp\mid e_{\fp}(M/\Q) = 2, \deg(\fp)=2 \}.
\end{gather*} 
Let $s=|X_1|$, $t=|X_2|$ and $g=|X_1|+|X_2|+|X_3|$. Then by \eqref{eq:efg}, $2s+2t+4(g-t-s)=h$ whence $g=\frac{h+2s+2t}{4}$. Thus the proof of Theorem~\ref{them:prime factorization}(A)(2) is reduced to determining $s$ and $t$. 

Since $e_\fp(L/M)=2$ and $f_\fp(L/M)=1$ for $\fp \in X_1$, there is a bijection
\begin{equation}\label{eq:f2e1}
\{ \fP | \text{ } I_{\fP}=\langle \tau \rangle \} \xleftrightarrow{1-1}{} X_1, \quad \fP \mapsto \fp=\fP \cap M,
\end{equation} 
Since $e_\fp(L/M)=1$ and $f_\fp(L/M)=2$ for $\fp \in X_2$, there is a bijection
\begin{equation}\label{eq:f1e2}
\{ \fP | \text{ } I_{\fP}\neq \langle \tau \rangle, D_\fP\supset \langle \tau \rangle \} \xleftrightarrow{1-1}{} X_2, \quad \fP \mapsto \fp=\fP \cap M.
\end{equation} 
For each $\fP$, by \eqref{eq:ef2} we have
	\begin{equation}\label{eq:dec_inertia}
\# I_\fP=2,  \quad  \# D_\fP(L/K)=2, \quad \# D_\fP =\#( I_\fP \cdot D_\fP(L/K))=4. 
\end{equation}	
The last equality is because $D_\fP(L/K)\subset \cl(\calO)$ but $I_\fP\not\subset \cl(\calO)$ as $\fP$ is unramified in $L/K$.

\medskip

\noindent (I) The case $e_p(F^{+}/\Q)=1$.

\medskip
Suppose first $p$ splits completely in $F^{+}$, i.e.  $f_p(F^{+}/\Q)=1$. Then 
\begin{equation}\label{eq:ID1}
I_\fP\subset D_\fP \subset \Gal(L/F^{+}) =\cl(\calO)^2\rtimes \langle \tau\rangle.
\end{equation}
Since $I_\fP\not\subset \cl(\calO)$, it follows that  $I_\fP =\{ 1, \sigma^2_0 \tau\}$ ($\sigma_0 \in \cl(\calO)$) is conjugate to $\langle \tau\rangle$. Thus there exists some $\fP'$ such that $I_{\fP'}=\langle \tau \rangle$.  Then we obtain $s=2^{\mu}/4=2^{\mu-2}$ Lemma~\ref{lem:central}(1). Now by \eqref{eq:ID1}, $D_{\fP'}(L/K) = \{1,\sigma^2\} $  for some $\sigma \in \cl(\calO)$ with $\sigma^2\in \cl(\calO)[2]$. Therefore, $D_{\fP'}=\{ 1, \tau, \sigma^2, \sigma^2\tau\}$ with $I_{\fP'}=\langle\tau\rangle$. Then
\[ I_{\sigma^{-1}(\fP')} = \langle \sigma^2 \tau \rangle \neq \langle \tau \rangle \quad \text{ and } \quad D_{\sigma^{-1}(\fP')} \supset \langle \tau \rangle.\]
Applying Lemma~\ref{lem:central}(3) with respect to $\sigma^{-1}(\fP')$ gives $t=2^{\mu} /4= 2^{\mu-2}$.

Suppose next $p$ does not split completely in $F^{+}$, i.e. $f_p(F^{+}/\Q)=2$. Then $I_\fP\subset \Gal(L/F^{+})=\cl(\calO)^2 \rtimes \langle \tau\rangle$. By a similar discussion as in the previous case, there is some $\fP'$ such that  $I_{\fP'}=\langle \tau\rangle $. Hence by Lemma~\ref{lem:central}(2), we have $s=2^{\mu} /4 = 2^{\mu-2}$. By the assumption, $f_p(L/F^{+})=1$ and hence $D_\fP\not\subset \Gal(L/F^{+})$. In particular,  $D_\fP(L/K)=D_\fP \cap \Gal(L/K) 
 \not\subset \cl(\calO)^2 = \Gal(L/F^{+}) \cap \Gal(L/K)$. So $D_{\fP'}(L/K)=  \{1,\delta\}$ for some $\delta\notin \cl(\calO)\setminus \cl(\calO)^2.$ Then $D_{\fP'} = \{1,\tau, \delta, \delta\tau\}$ with $I_{\fP'}=\langle \tau\rangle$. Both $\delta$ and $\delta\tau$ are not conjugate to $\tau$ by Lemma~\ref{lem:G_L}. Since any $D_\fP$ is conjugate to $D_{\fP'}$, it follows that the sets in \eqref{eq:f1e2} are both empty whence $t=0$. 
 
\medskip
	
\noindent (II) The case  $e_p(F^{+}/\Q) =2$.	

\medskip 	
In this case we have $e_{\fp}(M/\Q)=2$ and $e_\fp(L/M)=1$ as $F^{+}\subset M$. This implies $s=0$. Note that for each $\fP$, $e_\fP(L/M)=1$ since $2=e_p(L/\Q)=e_\fP(L/M)e_{\fP\cap M}(M/\Q)$. Hence $I_\fP\neq\langle \tau\rangle$ for each $\fP$.

Now if $f_{p}(F/F^{+})=2$, then $f_\fP(L/F^{+})=2$ whence $D_\fP(L/F^{+}) = D_\fP \cap (\cl(\calO)^2 \rtimes \langle \tau\rangle) = \langle \sigma^2 \tau \rangle $ ($\sigma \in \cl(\calO)$) has order $2$.  So $D_{\fP'}(L/M) = \langle \tau \rangle$ and $\fP'$ is inert in $L/M$ where $\fP'=\sigma^{-1}(\fP)$. In a word, $I_{\fP'}\neq \langle \tau \rangle$ and $D_{\fP'}\supset \langle \tau \rangle$.
By Lemma~\ref{lem:central}(3), $t= 2^{\mu}/4=2^{\mu-2}$. 

If $f_p(F/F^{+})=1$, then $e_p(F/F^{+})=f_p(F/F^{+})=1$ which implies $D_p(F/F^{+})=\{1\}$. In general, we have an exact sequence of decomposition groups:
\[
1\to D_\fP(L/F) \to D_\fP(L/F^{+}) \to D_{\fP\cap F}(F/F^{+}) \to 1.
\]
It follows that $D_\fP(L/F^{+}) = D_\fP(L/F)$. Then $D_\fP(L/M) = D_\fP(L/F^{+}) \cap \Gal(L/M)\subset \Gal(L/F)\cap \Gal(L/M)=\{1\}$. This implies that for each $\fp$, $e_\fp(L/M)=f_\fp(L/M)=1$ whence $e_\fp(M/\Q)=f_\fp(M/\Q)=2$. Thus we have $t=0$. 
  \end{proof}

\begin{remark}
	The number of degree $1$ primes in the first part of Theorem~\ref{them:prime factorization}(A-1) for $p>3$ and $D>-\frac{4}{\sqrt{3}}\sqrt{p}$ was the result of  Xiao-Luo-Deng~\cite[Theorem 5]{GLY21}, however one can see from their proof that the bound for $D$ is not needed.
\end{remark}

In the following  we describe a special case of Theorem~\ref{them:prime factorization}.

\begin{corollary}~\label{coro:O(q)}
	Let $q$ be an odd prime, $K=\Q(\sqrt{-q})$,  $D\in\{-q,-4q\}$ and $h=h_D$. Suppose $\big(\frac{-q}{p}\big)\neq 1$.
	\begin{enumerate}[1]
		\item If $p=q$, then \[p\calO_M=\begin{cases}
		\fp_1\fp^2_1\cdots\fp^2_{\frac{h+1}{2}}, &\text{if}\ q\equiv 3 \bmod{4};\\
		\fp^2_1\cdots\fp^2_{\frac{h}{2}}, &\text{if}\ q\equiv 1 \bmod{4}.
		\end{cases}\] where all prime ideals $\fp_i$ are of degree 1.
		\item If $\big(\frac{-q}{p}\big)= -1$, then 
		\[p\calO_M = \fp_1\fp_2\cdots\fp_g\]
		with all $\fp_i$ having degree $1$ or $2$; moreover,
		\begin{enumerate}[i]
			\item if  $ q \equiv 3 (\bmod 4)$, then $g=\frac{h+1}{2}$ and
			there is precisely one prime of degree $1$;
			\item if  $ q \equiv 1 (\bmod 4)$, then $g=\frac{h+1+\big(\frac{q}{p}\big)}{2}$ and 
			there are precisely $1+\big(\frac{q}{p}\big)$ primes of degree $1$.	\end{enumerate}
	\end{enumerate}
	 \end{corollary}

\begin{proof}
	 In this situation, we have $F^{+}= \Q$ and $\mu=1$ if $q\equiv 3\bmod 4$, and $F^{+}=\Q(\sqrt{q})$ and $\mu=2$ if $q\equiv 1\bmod 4$. Corollary~\ref{coro:O(q)} then follows from Theorem~\ref{them:prime factorization}(1) directly.
\end{proof}

For completeness, we describe the simpler case that $p$ splits in $K$ to end this section.
\begin{proposition} 
Let the notation be as in Theorem~\ref{them:prime factorization}. Assume that $p$ splits in $K$. Let $\lambda$ be the order of any prime $\fp_K$ in $K$ above $p$ in $\cl(\calO^{(p)})$. Then
 \[p\calO_M = (\fp_1\cdots \fp_g)^{h_D^{(p)}} \] 
where $\deg(\fp_1)=\cdots=\deg(\fp_g)=\lambda$,  $g = h_{D^{(p)}}/\lambda$, and $h_D^{(p)}=h_D/h_{D^{(p)}}=(M:M^{(p)})$.
\end{proposition}

\begin{proof}
It suffices to prove  the case $p\nmid f$ by Proposition~\ref{prop:totrami}. In this case, $\calO^{(p)}=\calO$ and $M^{(p)}=M$, and $p$ is unramified in $L$. 
 Since $p$ splits in $K/\Q$ and $L=MK$, it follows that
$\fp$ splits in $L/M$ for any prime $\fp$ of $M$ above $p$. Hence $\deg(\fp)=\deg(\fP)$ where $\fP$ is any prime of $L$ above $\fp$. 
We have $\deg(\fP) = f_{\fP}(L/K)=\lambda$ by the isomorphism \eqref{eq:clgp_iso} and by the obvious fact that $\lambda$ is independent of the choice of $\fp_K$. This proves the case $p\nmid f$.  
\end{proof}


\section{Factorization of the Hilbert class polynomial $H_D(x)$ over $\F_p[x]$}
In this section, we shall study the factorization of $H_D(x)$ modulo primes $p$. 
We continue to use the notations defined in \S~2.  We also denote $\overline{H}_D(x)= H_D(x)\bmod{p}\in \F_p[x]$, and denote
		\begin{equation}
			 n_D=[\calO_M: \Z[j_D]],\qquad  
		i_p=v_p(n_D)	
	\end{equation}
where $v_p$ means the $p$-adic valuation. Since $N_{M/Q}(H_D'(j_D))=\disc H_D(x)=n_D^2D_M$, we have
\begin{equation}\label{eq:disc H_D}
	v_p\left(N_{M/Q}\big(H_D'(j_D)\big)\right)=2i_p+v_p(D_M).
\end{equation}

\subsection{The case $p\nmid  n_D $.}  In this case, according to \cite[Theorem 4.8.13]{Coh96}, the factorization of $H_D(x)$ over $\F_p[x]$ is determined by the factorization of $p\calO_M$. 

\begin{theorem}~\label{them:facto. of class [polynomial]}  Fix $D,\ \calO$ and hence the fields $K, L, M, F$ and $F^{+}$. Assume $p\nmid  n_D f$.
	
	$(A)$ Suppose $p$ dose not split in $K$.
	\begin{enumerate}[1]
	
		\item If $p\nmid D_K$, then
		 \[\overline{H}_D(x) = \begin{cases}
		(x-j_1)\cdots(x-j_{2^{\mu-1}})p_1(x)\cdots p_{\frac{h-2^{\mu-1}}{2}}(x), &\text{if}\ p\ \text{splits completely  in}\ F^{+};\\
		p_1(x)\cdots p_{\frac{h}{2}}(x), &\text{otherwise}.
		\end{cases}\]

		\item If $D\in \{-p,-2p,-4p\}$, then 
		\[\overline{H}_D(x)=\begin{cases}
		(x-j_{1})(x-j_2)^{2}\cdots (x-j_{\frac{h+1}{2}})^{2}, &\text{if}\ D\in \{-4, -8, -p, -4p\mid p\equiv 3\bmod{4}\};\\
		(x-j_1)^{2}\cdots (x-j_{\frac{h}{2}})^{2}, &\text{if}\ D=-4p\ \text{and}\ p\equiv1 \bmod 4.
		\end{cases}\]

\item If $p\mid D_K$, $D\notin \{-p,-2p,-4p\}$ and $p$ is unramified in $F^{+}$, then 
			\[\overline{H}_D(x) = \begin{cases}
			p_1(x)\cdots p_{2^{\mu-2}}(x)(x-j_{1})^{2}\cdots(x-j_{2^{\mu-2}})^{2}p_{2^{\mu}+1}(x)^{2}\cdots p_{\frac{h+3\cdot2^{\mu}}{4}}(x)^{2}, &\text{if}\ f_p(F^{+}/\Q)=1;\\
			p_1(x)\cdots p_{2^{\mu-2}}(x)p_{2^{\mu-2}+1}(x)^{2}\cdots p_{\frac{h+2^{\mu-1}}{4}}(x)^{2}, &\text{otherwise}.
			\end{cases}\]
\item If $p\mid D_K$, $D\notin \{-p,-2p,-4p\}$ and $p$ is ramified in $F^{+}$, then 
\[\overline{H}_D(x) = \begin{cases}
(x-j_{1})^{2}\cdots(x-j_{2^{\mu-2}})^{2}p_1(x)^{2}\cdots p_{\frac{h-2^{\mu-1}}{4}}(x)^{2}, &\text{if}\ f_p(F/F^{+})=1;\\
p_1(x)^{2}\cdots p_{\frac{h}{4}}(x)^{2}, &\text{otherwise}.
\end{cases}\]
	\end{enumerate}
In each case above,  all $j_*$ are distinct elements in $\F_p$, and all $p_*(x)$ are distinct monic irreducible polynomials of degree $2$ over $\F_p$.

$(B)$ Suppose $p$ splits in $K$. Let $\lambda$ be the order of any prime $\fp_K$ in $K$ above $p$ in $\cl(\calO)$. Then 	\[\overline{H}_D(x)=f_1(x)\cdots f_{h/\lambda}(x),\] where all $f_*(x)$ are distinct monic irreducible polynomials of degree $\lambda$ over $\F_p$.
\end{theorem}

\begin{remark} If  $p\nmid n_D$ but $p\mid f$,  by Theorem~\ref{them:prime factorization}(B),  the factorization of $\overline{H}_D(x)$ is just  replacing $h$ and $\mu$ in the above Theorem by $h_{D^{(p)}}$ and $\mu^{(p)}=\log_2(\# (\cl(\calO^{(p)})/\cl(\calO^{(p)})^2))+1$, and then raising every factors in the right hand side to the $h^{(p)}_D$-power.
\end{remark}

\subsection{The case $p\mid  n_D $} In this case,  the factorization of $\overline{H}_D(x)$ over $\F_p[x]$ cannot follow directly from the factorization of $p\calO_M$. In order to saying something about its factorization without knowledge of $H_D(x)$, we need firstly to figure out the multiple factors of $\overline{H}_D(x)$. In this subsection, we assume $p$ does not split in $K$ and $p\nmid f$. 

For a prime ideal $\fP$ above $p$ in $L$, we count the number $m$ of conjugates $j_D'$ of $j_D$(we mean $j_D'\neq j_D$) such that $v_{\fP}(j_D-j_D')>0$. Then $j_D \bmod \fP$ is a $(m+1)$-multiple root of $\overline{H}_D(x)$ in $\F_{p^2}=\calO_L/\fP$. Assume that $E, E'$ are elliptic curves defined over $L$ such that $j(E)=j_D$ and $j(E')=j_D'$. Let $W$ be the completion of the maximal unramified extension of the valuation ring of $v_\fP$, and $\pi$ a uniformizer of $W$, and $\Iso_{W/\pi^n}(E,E')=\{\phi:E\xrightarrow{\sim} E' \bmod \pi^n\}$. The set $\Iso_{W/\pi^n}(E,E')$ is finite of order $0,2,4,6$ for $p\geq 5$.  From~\cite[Proposition 2.3]{Zagier85}, we have 
\begin{equation}\label{eq:v(theta-theta')}
v_\fP(j_D-j_D')=v_{\pi}(j_D-j_D')=\sum\limits_{n\geq 1} i_{\fP}(n)\ \text{and}\ i_{\fP}(n)=\frac{\#\Iso_{W/\pi^n}(E,E')}{2}.
\end{equation}

Suppose $\sigma_{\fb}(j_D)=j_D'$ for some $\sigma_{\fb}\in\cl(\calO)$, i.e., $E'=[\fb]*E$. The set $\Iso_{W/\pi^n}(E,E')$ consists of elements of reduced norm $1$ in $\Hom_{W/{\pi^n}}(E,E')$ which is isomorphic to $\End_{W/{\pi^n}}(E)\fb$ as an $\End_{W/{\pi^n}}(E)$ module in $B_{p,\infty}$. Dorman\cite{Dorman} described $\End_{W/{\pi^n}}(E)$ explicitly for any fundamental discriminant $D$ and showed in \cite[Lemma 4.8 and 4.11 ]{Dorman} that every isomorphism $\phi\in \Iso_{W/\pi^n}(E,E')$ induces two integral invertible ideals $\cc, \fd$ of $\calO$ satisfying:
\begin{equation}~\label{eq:solution}
\begin{split}& N(\cc)+p^{2n-1}N(\fd)=|D|, \ \text{if}\ p\ \text{is\ inert\ in}\ \calO; \\
& N(\cc)+p^{n-1}N(\fd)=|D|, \ \text{if}\ p\ \text{is\ ramified\ in}\ \calO. \end{split}
\end{equation}
Here $\cc,\fd$ are determined by $\fb$ and $\fP$. By counting the solutions of equation ~\eqref{eq:solution} over all $n\geq 1$ and $\fb\in \cl(\calO)\setminus\{1\}$, Dorman finally gave a formula to compute $v_{p}(N_{M/\Q}(H_D'(j_D)))$ whence $i_p$ by \eqref{eq:disc H_D}. Lauter and Viray\cite{lat} generalized  Dorman's description of  $\End_{W/{\pi^n}}(E)$ to an arbitrary discriminant $D$ such that $p\nmid f=[\calO_K:\calO_D]$ by substituting $D_K$ for $D$ and giving concrete definitions of related parameters, but they didn't consider the computation of $i_p$.
By Lauter and Viray's work, the $D$  in \eqref{eq:solution} can take any discriminant. Then
\begin{enumerate}[1]
	\item if $D>-p$, then $v_\fP(j_D-j_D')=0$;
	\item  if $p$ is inert (resp. ramified) in $K$ and $-p^3 < D < -p$ (resp. $-p^2 < D < -p$), then 
	\begin{equation}\label{eq:iP(1)}
	v_{\fP}(j_D - j_D')=i_{\fP}(1)=
	\begin{cases}
	3,\ \text{if}\ j_D \equiv j_D' \equiv 0 \bmod {\fP};\\
	2,\ \text{if}\ j_D \equiv j_D' \equiv 1728 \bmod {\fP};\\
	1,\ \text{if}\ j_D\equiv j_D' \not\equiv 0,1728 \bmod {\fP};\\
	0,\ \text{if}\ j_D\not\equiv j_D' \bmod {\fP}.
	\end{cases}
	\end{equation}

\end{enumerate}
The equation~\eqref{eq:iP(1)} tells us two obvious facts: 1) if $v_{\fP}(j_D-j_D')>0$ and $v_{\fP}(j_D-j_D'')>0$ for two distinct conjugates $j_D', j_D''$, then both valuations are equal to $i_{\fP}(1)$ which is determined only by $j_D\bmod \fP$; 2) if $\sigma\in\Gal(L/K)$ such that $v_{\fP}(\sigma(j_D)-\sigma(j_D'))>0$ for some $j_D'$, then
\begin{equation}\label{eq:i_p=i_sigma(p)}
	v_{\fP}(\sigma(j_D)-\sigma(j_D'))=v_{\sigma^{-1}(\fP)}(j_D-j_D')=i_{\sigma^{-1}(\fP)}(1).
\end{equation}

Now we state the results about the multiple factors of $\overline{H}_D(x)$ over $\F_p[x]$ when $i_p$ is small. 
\begin{theorem}~\label{them:ip<4}
	Assume $p\geq 5$, $D> -p^3$ and $p\nmid D$. If $1\leq i_p=v_p(n_D)\leq 3$, then the multiplicities of irreducible factors of $\overline{H}_D(x)$ are all $\leq 3$. More precisely, the multiple roots of $\overline{H}_D(x)$ can be described as follows:
	\begin{enumerate}[1]
		\item  if $i_p=1$, there is exactly one double root in $\F_{p}\setminus\{0,1728\}$;
		\item  if $i_p=2$, there are either two distinct double roots in $\F_{p^2}\setminus\{0,1728\}$ or one double root $1728$;
		\item if $i_p=3$, then one of the following cases happens:
		\begin{enumerate}[i]
			\item  there are exactly three double roots in $\F_{p^2}\setminus\{0,1728\}$;
			\item  there are exactly two double roots: $1728$ and another in $\F_{p}\setminus\{0,1728\}$;
			\item  there is exactly one double root $0$;
			\item  there is exactly one triple root in $\F_p\setminus\{0,1728\}$.
		\end{enumerate} 
	\end{enumerate}
\end{theorem}
\begin{proof}
	Suppose first $\fP$ is a prime ideal in $L$ above $p$.  Since $p\nmid D$, $p$ is unramified in $L$. Thus $v_p(D_M)=0$. By~\eqref{eq:disc H_D},
	\[4i_p=v_p\left(N_{L/\Q}(H'_D(j_D))\right)=v_{\fP}\left(N_{L/\Q}(H'_D(j_D))\right)=2\sum\limits_{\fP\mid p}v_{\fP}(H_D'(j_D)),\]
	where the sum is taken over all prime ideals $\fP$ in $L$ above $p$; i.e.,
	\begin{equation}\label{eq:i_p to v_P}
	2i_p=\sum\limits_{\fP\mid p}v_{\fP}(H_D'(j_D)).
	\end{equation}
	
	Suppose next $n=v_{\fP}\big(H'_D(j_D)\big)\geq 1$. Then there exist $\frac{n}{i_{\fP}(1)}$ conjugates $j_D'$ such that $v_{\fP}(j_D-j_D')=i_{\fP}(1)\geq 1$ by \eqref{eq:iP(1)}. Thus $j_D\bmod \fP$ is a multiple root of $\overline{H}_D(x)$ in $\F_{p^2}$ with multiplicity $\frac{n}{i_{\fP}(1)}+1$. Assume $\sigma\in\Gal(L/\Q)$ such that $j_D'=\sigma^{-1}(j_D)$. Let $\fP'=\sigma(\fP)$. Then 
	\begin{equation}\label{eq:111}
	v_{\fP'}(H'_D(j_D))=v_{\fP}(H'_D(j_D'))=v_{\fP}(H'_D(j_D))=n.
	\end{equation}
	The second equality follows from $v_{\fP}(j_D-\theta)=v_{\fP}(j'_D-\theta)$ for any conjugate $\theta\neq j_D, j'_D$. Thus a multiple root $j_D \bmod \fP$ contributes $\frac{n}{i_{\fP}(1)}+1$ prime ideals in the summation of \eqref{eq:i_p to v_P} which have the same valuation at $H'_D(j_D)$, i.e., the multiple root $j_D\bmod \fP$ contributes a value of $\frac{n^2}{i_{\fP}(1)}+n$ in the right hand of \eqref{eq:i_p to v_P}. Obviously,
	\begin{equation}\label{eq:limitation}
	\frac{n^2}{i_{\fP}(1)}+n\leq 2i_p.
	\end{equation}
	
	Moreover, suppose $j''_{D} \bmod\fP\neq j_D\bmod \fP$ is another root of $\overline{H}_D(x)$ in $\F_{p^2}:=\calO_L/\fP$, and $\sigma''(j''_D)=j_D$ for some $\sigma''\in\Gal(L/K)$. Then $\sigma''(\fP)$ doesn't belong to the previous $\frac{n}{i_{\fP}(1)}+1$ prime ideals set. Furthermore, if $j''_{D} \bmod\fP$ is a simple root of $\overline{H}_D(x)$, then $v_{\sigma''(\fP)}(H'_D(j_D))=v_{\fP}(H'_D(j''_D))=0$; if $j''_{D} \bmod\fP$ is a multiple root of $\overline{H}_D(x)$, assume $n_1=v_{\sigma''(\fP)}(H'_D(j_D))>0$, then $j''_{D} \bmod\fP$ contributes another value of $\frac{n_1^2}{i_{\sigma''(\fP)}(1)}+n_1$ by \eqref{eq:i_p=i_sigma(p)}.
	Consequently, the $h$ prime ideals in the summation of \eqref{eq:i_p to v_P} are partitioned into some disjoint sets according to the different roots of $\overline{H}_D(x)$ in $\F_{p^2}$. Moreover, all the value which derived from different roots of $\overline{H}_D(x)$ add up to $2i_p$.

	(1)	If $i_p=1$, then $n=i_{\fP}(1)=1$  by \eqref{eq:limitation}, and $\overline{H}_D(x)$ has only one double root $j_D \bmod \fP$ and it is in $\F_p\setminus\{0,1728\}$ . 
	
	\medskip	
	(2)	If $i_p=2$, then $n=i_{\fP}(1)=1$ or $n=i_{\fP}(1)=2$ if $j_D\equiv1728\bmod \fP$ by \eqref{eq:limitation}.
	\begin{itemize}
		\item [(i)] In the first case, $j_D \bmod \fP$ contributes the value of 2 in the right hand of \eqref{eq:i_p to v_P}. Thus there exists exactly another double root in $\F_{p^2}\setminus\{0,1728\}$ of $\overline{H}_D(x)$.
		\item [(ii)] In the second case, assume $j_D'=\sigma(j_D)$ such that $j_D'\equiv 1728 \bmod \fP$, then $4=v_{\fP}(H_D'(j_D))+v_{\sigma(\fP)}(H_D'(j_D))$. Thus $\overline{H}_D(x)$ has exactly one double root $1728$.
	\end{itemize}
	
	\medskip	
	(3)	
	If $i_p=3$, then the pair $(n, i_{\fP}(1))$ derived from different multiple roots of $\overline{H}_D(x)$ belongs to the following cases: 1) $n=i_{\fP}(1)=1$; 2) $n=i_{\fP}(1)=2$ if $j_D\equiv1728\bmod \fP$; 3) $n=i_{\fP}(1)=3$ if $j_D\equiv0\bmod \fP$; 4) $n=2, i_{\fP}(1)=1$. Then (3) in Theorem~\ref{them:ip<4} is obvious after some simple discussion as above.
\end{proof}

\begin{remark}
	 In Theorem~\ref{them:ip<4}, only inert primes $p$ in $K$ are considered, as our proof depends on the fact that $v_{p}\left(N_{M/\Q}(H_D'(j_D))\right)=2i_p$ is small. For ramified primes $p$ in $K$,  as $p\nmid f$, by \cite[Proposition 5.1]{Dorman}, $v_{p}\left(N_{M/\Q}(H_D'(j_D))\right)=2i_p+\frac{h-2^{\mu-1}}{2}$ if $p$ is the unique prime factor of $D$ which is $\equiv 3\bmod{4}$, or $2i_p+\frac{h}{2}$ if otherwise, which is more complicated and beyond our consideration.
\end{remark}

Now let us move back to the special case that $\calO$ is an imaginary quadratic suborder of the maximal order $\calO(p,q)$ or $\calO'(p,q)$ of $B_{p,\infty}$ of Ibukiyama introduced in \S~1.1. If $-p<D<0$, the factorization of $\overline{H}_{\calO}(x)$ can be derived directly from Corollary~\ref{coro:O(q)} as $p\nmid n_D$ which is the same as in \cite[Theorem 25]{CPV20}. In the following, we give the factorization of $\overline{H}_{\calO}(x)$ over $\F_p[x]$ when $-p^3<D<-p$. For simplicity, we only write down the case  $i_p=v_p(n_D)\leq 2$.

\begin{corollary}
	Let $p, q$ be two primes satisfying $\big(\frac{-q}{p}\big)=-1$ and $q\equiv3\bmod 4$. Let $\calO=\Z[\frac{1+\sqrt{-q}}{2}]$ or $\Z[\sqrt{-q}]$, and the corrsponding discriminant $D=-q$ or $-4q$, respectively. If $-p^3<D<-p$, then
	\[\overline{H}_D(x) = \begin{cases}
	(x-j_0)(x-j_1)^2p_1(x)\cdots p_{\frac{h_D-3}{2}}(x),\ & \text{if}\ i_p=1,\\
	(x-j_0)p_0(x)^2p_1(x)\cdots p_{\frac{h_D-2\deg p_0-1}{2}}(x),\ & \text{if}\ i_p=2,
	\end{cases}\]
	where $j_0\neq j_1\in \F_p$, $p_0(x)\in\F_p[x]$ is either of degree $2$  or $=x-1728$, and all other $p_{*}(x)$ are monic irreducible polynomials of degree $2$ in $\F_{p}[x]$.
\end{corollary}
\begin{proof}
   	For any simple factor $g(x)$ of $\overline{H}_D(x)$ over $\F_p[x]$, we have $(p,g(j_D))$ divides $p\calO_M$ with norm $p^{\deg g}$ by~\cite[Proposition 6.2.1]{Coh96}. Thus, we deduce that $\overline{H}_D(x)$ has exactly one simple linear factor and all other simple factors are of degree $2$ by Theorem~\ref{them:prime factorization}(1-i) and noting that in this case $h_D$ is odd. If $i_p\leq 2$, the multiplicity of all multiple roots of $\overline{H}_D(x)$ is $2$ according to Theorem~\ref{them:ip<4}-(1),(2), and the remaining part is obvious.
\end{proof}




\section{Key space of OSIDH}
In this section, we analyze the key space of OSIDH by combining the parameters given by Onuki for the protocol to work. We refer to \cite{Onuki, CK19} for the details of OSIDH. 

Let $\calO_0$ be an order of discriminant $D$, and $p$ a non-split prime in $K=\Q(\sqrt{D})$ such that $p\nmid [\calO_K:\calO_0]$. Let $\ell\neq p$ be another prime. For $n\geq 0$, let $\calO_{n}$ be the suborder of $\calO_0$ with conductor $\ell^n$ and $D_n$ be its discriminant. The OSIDH protocol is based on the commutative group action of $\cl(\calO_n)$ on $\rho\big( \Ell(\calO_n)\big)$(see  \S~\ref{sec:osidh}) as follows: starting from a public known pair $(E_n,\iota_n)\in\rho\big( \Ell(\calO_n)\big)$, Alice and Bob choose secret keys $[\fa]$ and $[\fb]$ respectively in $\cl(\calO_n)$; then they compute $[\fa]*(E_n,\iota_n)$ resp. $[\fb]*(E_n,\iota_n)$, and publish their results; Alice and Bob compute $[\fa]*[\fb]*E_n$ resp. $[\fb]*[\fa]*E_n$,  and take $j([\fa][\fb]*E_n)$ as the shared key. 

The computation of $[\fa]*(E_n,\iota_n)$ is related to $E_n[\fa]$ which may be defined over large field extension of $\F_p$ without some restrictions on $p$ such as those in SIDH or CSIDH. To avoid the time-consuming computation of the action $[\fa]*E_n$, Col\`o and Kohel proposed the modular $\ell$-isogeny ladders of length $n$ to get the $j$-invariant of $[\fa]*E_n$.
Assume that $[\fq]$ is the maximal generator of $\cl(\calO_n)$ with norm $q$. Onuki proved that the approach of computation of $\ell$-isogeny ladders by modular polynomials always works if $p>q|D_n|$.  By combining the factorization of $\overline{H}_{D_n}(x)$, we have

\begin{proposition}
	As notations are defined above. If $p>|D_n|=\ell^{2n}|D|$, then the size of the key space of OSIDH is $h_{D_n}$ which is less than $O(\sqrt{p}\log p)$. Moreover, let $F_n$ be the genus field of $\calO_n$, $F_n^{+}=F^{n}\cap \R$, and $\mu_n=\mu_{D_n}$. If $p$ splits completely in $F_n^{+}$, there exist $2^{\mu_n-1}$ $j$-invariants among the key space in $\F_p$.
\end{proposition}
\begin{proof}
	For prime $p>|D|\ell^{2n}=|D_n|$, we know that $\overline{H}_{D_n}(x)$ has no multiple roots in $\F_{p^2}$ by \S 4.2. By the reduction map, we have
	\[\{\text{roots\ of}\ H_{D_n}(x) \bmod p \ \text{in}\ \F_{p^2}\}= \rho\big( \Ell(\calO_n)\big)/\sim.\]
	Thus, the size of the key space is $h_{D_n}=\#\cl(\calO_n)$, which is less than $\sqrt{|D_n|}\log |D_n|$(see \cite[Exercise 5.27]{Coh96})  whence is up bounded by $O(\sqrt{p}\log p)$. 
	 The second assertion is equivalent to counting the number of $\F_p$-roots of $\overline{H}_{D_n}(x)$ which is clear by Theorem~\ref{them:facto. of class [polynomial]}(A)-(1).  
\end{proof}


\vskip 1cm
\newcommand{\etalchar}[1]{$^{#1}$}

\end{document}